\documentclass{amsart}

\usepackage{amsfonts,amssymb,bbm}
\usepackage{color}
\usepackage{graphicx,tikz}
\usepackage[font=footnotesize]{caption}
\usepackage{ifthen}
\usepackage{enumitem}
\usepackage[bookmarksnumbered,colorlinks, linkcolor=blue, citecolor=red,
bookmarks, breaklinks]{hyperref}
\usepackage[margin=10pt, skip=8pt, position=below]{caption}

\makeatletter

\IfFileExists{graphs.tex}{%

  \newcommand{\graphpath}[1]{##1}
}{%
  \newcommand{\graphpath}[1]{graphs/##1}
}

\newtheorem{thm}{Theorem}[section]

\newtheorem{conj}[thm]{Conjecture}

\newtheorem{prop}[thm]{Proposition}
\theoremstyle{definition}

\theoremstyle{remark}

\newtheorem{rem}[thm]{Remark}
\numberwithin{equation}{section}

\definecolor{umons-red}{RGB}{168, 0, 57}
\definecolor{umons-gray}{RGB}{150, 150, 150}
\definecolor{umons-turquoise}{RGB}{0, 171, 204}

\let\phi=\varphi
\let\le=\leqslant   
\let\ge=\geqslant   \let\geq=\geqslant
\let\subset=\subseteq

\newcommand{\R}{\mathbb R}

\newcommand{\IS}{\mathbb S}
\newcommand{\IF}{{\mathbb F}}
\newcommand{\IIR}{{\mathbb{IR}}}
\newcommand{\IIF}{{\mathbb{IF}}}
\newcommand{\I}{{\mathcal I}}

\newcommand{\Leb}{{\operator@font L}}
\newcommand{\rounddown}[1]{{\downarrow}(#1)}
\newcommand{\roundup}[1]{{\uparrow}(#1)}
\newcommand{\compactinterval}[3]{
  [#1{\scriptstyle #2}, #1{\scriptstyle #3}]}
\newcommand{\abs}[1]{\mathopen| #1\mathclose|}
\newcommand{\bigabs}[1]{\bigl| #1\bigr|}
\newcommand{\norm}[1]{\mathopen\| #1\mathclose\|}
\newcommand{\liff}{\Leftrightarrow}

\DeclareMathOperator{\Dom}{Dom}
\DeclareMathOperator{\e}{e}
\DeclareMathOperator{\width}{width}
\DeclareMathOperator{\low}{low}
\DeclareMathOperator{\high}{high}
\DeclareMathOperator{\diam}{diam}

\newcommand{\intervaloo}[1]{\mathopen] #1\mathclose[}
\newcommand{\intervalco}[1]{\mathopen[ #1\mathclose[}
\newcommand{\intervaloc}[1]{\mathopen] #1\mathclose]}
\newcommand{\intervalcc}[1]{\mathopen[ #1\mathclose]}
\newcommand{\bigintervalcc}[1]{\bigl[ #1\bigr]}

\newcommand{\ul}{\overline}
\newcommand{\ve}{\varepsilon}
\newcommand{\dl}{\underline}

\newcommand{\intd}{\,{\mathrm d}}

\newcommand{\wto}{\rightharpoonup} 

\newcommand{\code}[1]{{\texttt{#1}}}

\setlist[itemize,1]{leftmargin=3ex, topsep=0.1ex, itemsep=0.1ex}

\newenvironment{algo}[1][]{%
  \par\vspace{2ex}
  \noindent
  \ifthenelse{\equal{#1}{}}{\hspace*{5mm}}{\hspace*{10mm}}%
  \begin{minipage}{0.9\linewidth}
    \addtolength{\baselineskip}{1pt}%
    \newenvironment{block}{\noindent\hspace*{5mm}%
      \begin{tabular}{l}}{\end{tabular}}%
    \ttfamily
    \ifthenelse{\equal{#1}{}}{}{%
    \noindent\hspace*{-5mm} \hbox{#1}\\}%
  }{%
  \end{minipage}%
  \par\vspace{2ex}}

\makeatother

\begin{document}

\begin{abstract}
  In this article, using a Lyapunov-Schmidt reduction on an asymptotic
  Nehari manifold and verified computations,
  we prove that the least energy nodal solutions to
  Lane-Emden equation $-\Delta u = \abs{u}^{p-2} u$ with zero
  Dirichlet boundary conditions on a square are odd with
  respect to one diagonal and even with respect to the other one when
  $p$ is close to $2$.
\end{abstract}

\title[A computer assisted proof of the symmetries of l.e.n.s.]{%
  A computer assisted proof of the symmetries of
  least energy nodal solutions on squares}

\subjclass[2020]{35J20 (Primary) 35B06, 65G40 (Secondary)}

\author[Salort]{Ariel Salort}

\address[A. Salort]{Departamento de Matem\'atica, FCEyN - Universidad de Buenos Aires and
\hfill\break \indent IMAS - CONICET
\hfill\break \indent Ciudad Universitaria, Pabell\'on I (1428) Av. Cantilo s/n. \hfill\break \indent Buenos Aires, Argentina.}
\email{asalort@dm.uba.ar}
\urladdr{http://mate.dm.uba.ar/~asalort}

\author[Troestler]{Christophe Troestler}
\address[C. Troestler]{
  D\'epartment de Math\'ematique,
  Universit\'e de Mons, Place du Parc, 20, B-7000 Mons, Belgium
}

\email{christophe.troestler@umons.ac.be}
\urladdr{http://staff.umons.ac.be/christophe.troestler}
\thanks{The authors are partially supported by the project
  ``Existence and asymptotic behavior of solutions
  to systems of semilinear elliptic partial differential equations''
  (T.1110.14)
  of the \textit{Fonds de la Recherche Fondamentale Collective},
  Belgium.}
\maketitle

\section{Introduction}

Given a connected domain $\Omega\subset \R^N$, $N\geq 2$ and
$2<p<2^*:=\frac{2N}{N-2}$ ($+\infty$ if $N=2$), we consider the
so-called Lane-Emden problem with Dirichlet boundary conditions
\begin{align} \label{LEP}
  \begin{cases}
    -\Delta u = |u|^{p-2}u &\quad \text{ in } \Omega,\\
    u=0 &\quad \text{ on } \partial \Omega.
  \end{cases}
\end{align}
In this article we aim on symmetry properties of solutions of
\eqref{LEP} for values of $p$ close to $2$. The literature on this
subject is vast. In the early '70s Ambrosetti and Rabinowitz
\cite{AmRa73} generalize the work of Nehari \cite{Nehari59} to
prove that \eqref{LEP} has a positive ground state
solution, which, in~\cite{GiNiNi79} was shown to inherit all symmetries
of convex domains.
These solutions are unique when $p$ is close to $2$ and ``usually''
non-degenerate as established by Dancer~\cite{Dancer03}.

Later on, in the '90s, solutions with minimum energy among
sign-changing ones (or
\emph{least energy nodal solutions} hereafter referred to as
\emph{l.e.n.s.\@}) were proved to exists by Castro,
Cossio and Neuberger~\cite{CaCoNe} by minimization on a
``nodal Nehari set.''
A natural question is whether l.e.n.s.\@ inherit (some) symmetries of
the domain. There are several
result in this direction. Aftalion and Pacella~\cite{AfPa} proved
that, on a ball, a l.e.n.s.\@ cannot be radial.
Bartsch, Weth and Willem~\cite{BaWeWi} established that,
on radial domains, l.e.n.s.\@ are foliated Schwarz symmetric and,
in particular, even w.r.t.\@ $N-1$ orthogonal hyperplanes.
In \cite{BBGS}, symmetries on more general domains
were studied, as well as the asymptotic behavior of solutions when $p
\to 2$.
A similar asymptotic analysis was performed by Grossi who showed that,
from any eigenvalue of multiplicity $m$ of the linear problem
$-\Delta u = \lambda u$ with zero Dirichlet boundary conditions,
$m$ branches of solutions emanate~\cite{Grossi09}.

It is worth to mention that the arguments leading to uniqueness
and the symmetry properties of
l.e.n.s.\ of \eqref{LEP} for values of $p$ close to $2$
strongly depends on the simplicity of
$\lambda_2$. Indeed, when $\lambda_2$ is simple, uniqueness of the limit
solution follows by an application the \emph{implicit function
  theorem}. In this case, for $p$ close to~$2$, l.e.n.s.\ are unique (up
their sign) and possess the same symmetry as eigenfunctions
corresponding to $\lambda_2$.  In particular, on a rectangle,
the nodal line is the small median~\cite{BBGS}.

However, the approach is different when $\lambda_2$ has multiplicity.
When $\Omega$ is a ball, as all eigenfunctions for $\lambda_2$ can be
obtained from a one-dimensional subspace by rotations, one is still
able to apply the implicit function theorem and deduce that l.e.n.s.\
are odd w.r.t.\ an hyperplane passing through the center of the ball
and invariant by rotations around the axis orthogonal to that
hyperplane \cite[Theorem~3]{BBGS}.  In particular, their nodal surface
is a $N-1$-dimensional ball.
For general domains, the implicit function theorem can no longer be
applied and symmetries of l.e.n.s.\
for values of $p$ near $2$ are less understood.  Even for
simple domains such as planar squares, many questions remain
open. Particularly, numerical simulations carried out in~\cite{BBGS}
support the following assertion:

\begin{conj}
  \label{conj}
  If $\Omega\subset \R^2$ is a square and $p$ is close to $2$,
  l.e.n.s.\ are even with respect to a diagonal and
  odd with respect to the orthogonal diagonal.  In particular, their
  nodal line is a diagonal.
\end{conj}
Despite the fact that everything is explicit,
a ``paper and pencil proof'' of that fact seems not to
be clear.  The main aim of this article is to
prove this conjecture by using a computer-assisted proof approach.

The manuscript is organized as follows. In section \ref{sec.prelim} we
introduce some preliminary results. In Section
\ref{sec.idea}, we explain the main ingredients of our arguments.
Section~\ref{sec.proof} is devoted to the description and implementation our
computer-assisted proof and in Section~\ref{sec.main} we
put it all together to establish the symmetry property.

\section{Preliminary  results}
\label{sec.prelim}

In this section we introduce some notations and preliminary
results. We first recall a variational formulation for
sign-changing solutions of \eqref{LEP} as well as the limit equation
they fulfill as $p$ approaches $2$. Finally we present some abstract
results on symmetries of solutions for values of $p$ close to $2$.

\subsection{Variational formulation}

In order to study solutions of \eqref{LEP}, and to avoid further
normalizations, we consider the following problem
\begin{align} \label{P} \tag{$\mathcal{P}_p$}
  \begin{cases}
    -\Delta u = \lambda_2|u|^{p-2}u &\quad \text{ in } \Omega,\\
    u=0 &\quad \text{ on } \partial \Omega,
  \end{cases}
\end{align}
where $\lambda_2$ is the second eigenvalue of $-\Delta$ with zero Dirichlet
boundary conditions.
It is straightforward that symmetry properties of solutions
to~\eqref{LEP} and \eqref{P} are the same since $u$ solves \eqref{LEP}
if and only if $\lambda_2^{1/(2-p)}u$ solves \eqref{P}.
Weak solutions to~\eqref{P} are critical points of the
functional $\mathcal{J}_p:H^1_0(\Omega)\to\R$ defined by
\begin{equation*}
  \mathcal{J}_p(u)
  =\frac12 \int_\Omega |\nabla u|^2\intd x
  - \frac{\lambda_2}{p} \int_\Omega |u|^p\intd x.
\end{equation*}
This functional has Fr\'echet derivative given by
\begin{equation*}
  \forall v\in H^1_0(\Omega),\qquad
  \mathcal{J}'_p (u)[v]
  = \int_\Omega \nabla u\cdot \nabla v\intd x
  - \lambda_2 \int_\Omega |u|^{p-2}uv\intd x .
\end{equation*}
We remark that standard regularity theory arguments imply that
weak solutions to~\eqref{P} belong to $C^2_0(\Omega)\cap C(\overline\Omega)$
and hence classical solutions.

Clearly the zero function solves \eqref{P}.
All non-zero critical points of $\mathcal{J}_p$ belong to the Nehari
manifold
\begin{equation*}
  \mathcal{N}_p
  :=\left\{u\in H^1_0(\Omega)\setminus\{0\}:
    \mathcal{J}'_p (u)[u]  =0 \right\}
\end{equation*}
and all sign-changing solutions belong to the nodal Nehari set
\begin{equation*}
  \mathcal{M}_p
  := \left\{ u\in H^1_0(\Omega): u^\pm \in \mathcal{N}_p \right\}
\end{equation*}
where as usual $u^\pm:=\max\{\pm u,0\}$. Given a function $u\in
H^1_0(\Omega)$ with $u^\pm\neq 0$,  $u\in \mathcal{M}_p$
if and only if $\int_\Omega |\nabla u^\pm|^2\intd x
= \lambda_2 \int_\Omega |u^\pm|^p \intd x$.

It can be shown~\cite{BaWeWi, CaCoNe} that the minimum of
$\mathcal{J}_p$ on $\mathcal{M}_p$ is achieved and that all minimizers
are nodal solutions of \eqref{P}---referred to as \emph{least energy
  nodal solutions}---with exactly two nodal domains.

\subsection{The asymptotic equation}
\label{sec:lim-eq}

Let us analyze the behavior of l.e.n.s.\@ solutions of \eqref{P} as $p\to 2$.
To start, any family $(u_p)_{p> 2}$ of l.e.n.s.\ of \eqref{P}
can be proved to be bounded in $H^1_0(\Omega)$ and bounded
away from $0$ (see \cite[Lemma~4.1 and Lemma~4.4]{BBGS}).  This
implies that any weak accumulation point $u_*$
of $(u_p)_{p> 2}$ verifies a limit equation \cite[Theorem~4.5]{BBGS}:
if $u_{p_n} \wto u_*$
weakly in $H^1_0(\Omega)$ as $p_n\to 2$, then $u_{p_n} \to u_*$ in
$H^1_0(\Omega)\setminus\{0\}$ and $u_*$ fulfills
\begin{align} \label{limit.eq}
  \begin{cases}
    -\Delta u_* =\lambda_2 u_* &\quad \text{ in } \Omega,\\
    u_*=0&\quad \text{ on } \partial \Omega.
  \end{cases}
\end{align}
Thus $u_* \in E_2\setminus\{0\}$ where $E_2$ denotes the eigenspace
related to the second eigenvalue $\lambda_2$.
Moreover, $u_*$ minimizes the \emph{asymptotic functional}
$\mathcal{J}_*$ over the \emph{asymptotic Nehari manifold}
$\mathcal{N}_*$, where
\begin{gather}
  \label{j.star}
  \mathcal{J}_*: E_2 \to \R : u \mapsto \mathcal{J}_*(u)
  = \frac{\lambda_2}{4}\int_\Omega u^2- u^2\log u^2 \intd x ,
  \\[1\jot]
  \notag 
  \mathcal{N}_*
  = \{ u \in E_2\setminus\{0\} : \mathcal{J}_*' (u)[u] =0\},
\end{gather}
and $t^2\log t^2$ is extended continuously by $0$ at $t=0$.
In particular $u_*$ is a critical point of $\mathcal{J}_*$.
It is easy
to see that
\begin{equation}
  \label{eq:J.star'}
  \forall v\in E_2,\qquad
  \mathcal{J}_*' (u)[v]
  = -\lambda_2 \int_\Omega v u \log \abs{u}\intd x.
\end{equation}
Any nontrivial critical point of $\mathcal{J}_*$ belongs
to~$\mathcal{N}_*$.
This manifold $\mathcal{N}_*$ is compact and such that
$u\in\mathcal{N}_*$ if and only if $u \ne 0$ and
\begin{equation*}
  \int_\Omega u^2 \log \abs{u}\intd x=0
  \quad \text {or equivalently  iff} \quad
  \mathcal{J}_* (u) = \frac{\lambda_2}{4}\int_\Omega u^2\intd x.
\end{equation*}
Let us now show that $u_*$ satisfies an equivalent minimization
problem.  This problem will be the one used in our computer assisted
proof.
First, as for the classical Nehari manifold, one has that
\begin{equation} \label{eq.t.star}
  \inf_{u \in \mathcal{N}_*} \mathcal{J}_*(u)
  = \inf_{ u \in \IS} \max_{t>0} \mathcal{J}_* (tu),
\end{equation}
where $\IS$ is the $\Leb^2$-unit sphere of $E_2$.
Indeed, observe that, for any $u\in E_2\setminus \{0\}$, the ray
$\{tu : t>0\}$ intersects $\mathcal{N}_*$ at a single point
$t^*_u \, u$ maximizing $\mathcal{J}_*$.  The value of
$t^*_u \in \intervaloo{0,+\infty}$ is easily computed:
\begin{equation}
  \label{eq:proj:N*}
  t^*_u = \exp\left( - \frac{\int_\Omega u^2 \log \abs{u}\intd x}{
      \int_\Omega u^2 \intd x} \right).
\end{equation}
In view of the above expression, given $u\in \IS$, we get
\begin{equation*}
  \mathcal{J}_*(t^*_u u)
  = \frac{\lambda_2}{4}  \int_\Omega (t^*_u u)^2\intd x
  = \frac{\lambda_2}{4} \e^{- 2\int_\Omega u^2 \log \abs{u}\intd x }
\end{equation*}
and consequently,  since  the map $t\mapsto e^{2 t}$ is increasing,
\eqref{eq.t.star} is equivalent to
\begin{equation}
  \label{eq:minimize:S}
  \text{minimize }
  \mathcal{S}_*(u) := -\int_\Omega u^2 \log\abs{u}\intd x
  \text{ on the } \Leb^2\text{-sphere } \IS \text{.}
\end{equation}
Note that for any $u \in \IS$ and $r > 0$,
\begin{equation}
  \label{eq:minimize:rS}
  \mathcal{S}_* (ru)
  = r^2 \mathcal{S}_*(u) - r^2 \log r
\end{equation}
so, instead of \eqref{eq:minimize:S}, one can equivalently minimize
$\mathcal{S}_*$ on the sphere $r\IS$ of radius $r$.

The function $\mathcal{S}_*$ can be seen as the \emph{entropy}
associated to the density $|u|^2$ and accumulation points $u_*$ of the
minimal energy nodal solutions are multiples of the eigenfunctions
with minimal
entropy.
Now that the limit functions $u_*$ of l.e.n.s.\ sequences are
characterized, our following concern is to understand how their
symmetries yield symmetries of the $u_p$ for $p$ close to~$2$.

\subsection{Symmetries}

As already noted, when studying symmetries there are two
possibles scenarios depending of the simplicity of $\lambda_2$.

When $\dim E_2=1$, using
the \emph{implicit function theorem}, it follows that \eqref{P}
has a unique solution (up its sign) for $p \approx 2$.
Therefore \cite[Theorem~2]{BBGS}, there exists $p^*$ sufficiently
close to $2$ such that, for any reflection $R$ such that
$R(\Omega)=\Omega$, a function $u_2 \in E_2\setminus\{0\}$ is even or
odd w.r.t.\ $R$ if and only if $u_p$ possesses the same invariance for
$p \in \intervaloc{2, p^*}$.
For instance, when $\Omega$ is a rectangle, since
$\lambda_2$ is simple, l.e.n.s.\ to
\eqref{P}, with $p \approx 2$, are odd w.r.t.\ the small
median of the rectangle (so their nodal line is that
median) and even w.r.t.\ the large median.

When $\lambda_2$ is not simple, with the exception of radial domains,
the questions of uniqueness and symmetries become more delicate.
Symmetries of l.e.n.s.\ to \eqref{P} for $p$ close to $2$ were
addressed in \cite[Theorem~3.6]{BBGS} by the following abstract
result. Consider a family of groups $(G_\alpha)_{\alpha\in E_2}$
acting on $H^1_0(\Omega)$ in a such way that, for every
$\alpha \in E_2$, $g\in G_\alpha$, $p$ close to $2$, and
$u\in H^1_0(\Omega)$, the following holds:
\begin{equation*}
  \text{(i)}\; g(E_2)=E_2, \quad
  \text{(ii)}\; g(E_2^\perp)=E_2^\perp,\quad
  \text{(iii)}\; g\alpha = \alpha, \quad
  \text{(iv)}\;  \mathcal{J}_p(gu)=\mathcal{J}_p(u).
\end{equation*}
Then, for all $M>0$, if $p$ is close to $2$, any l.e.n.s.\
$u_p\in \{u\in B_M: P_{E_2}(u)\notin B_{1/M}\}$ of \eqref{P} is
invariant under the isotropy group
$G_{\alpha_p}=\{g\in G : g\alpha_p=\alpha_p\}$ of
$\alpha_p=P_{E_2}u_p$, the orthogonal projection of $u_p$ on $E_2$.
In particular, if $G_\alpha$ describes the symmetries (or antisymmetries)
of $\alpha$, for $p$ close to $2$, $u_p$ possesses the same symmetries as
its orthogonal projection $\alpha_p$. For instance, when $\Omega$ is a
square and $p$ is close to $2$, l.e.n.s.\ of \eqref{P} are odd with
respect to the center of the square.

Moreover, as we anticipated, in \cite[Conjecture 5.4]{BBGS} numerical
computations suggest that $u_*$ is odd w.r.t.\ a diagonal and even w.r.t.\
the other diagonal.
It seems natural to think that these symmetries extend to $u_p$ for
$p$ close to~$2$, whence conjecture~\ref{conj}.  This is the content of
our main result.

\begin{thm}\label{main}
  If $\Omega\subset \R^2$ is a square and $p$ is close to $2$,
  l.e.n.s.\ are even with respect to a diagonal and
  odd with respect to the orthogonal diagonal.
\end{thm}

Unfortunately the above mentioned abstract result is not powerful enough for
this purpose.  Indeed, the computer assisted proof below will
characterize the symmetries of $u_* \in \mathcal{N}_*$ but that does
not readily implies that $P_{E_2} u_p = u_*$, so it is not clear that
$u_p$ enjoys the same symmetries as $u_*$.
To do that, the implicit function theorem by needs to be replaced by a
Lyapunov-Schmidt reduction.  Such approach was done by
Grossi and we extract here (with our notations)
the part of his results that is useful for our
purposes~\cite{Grossi09}.

\begin{thm}[Lyapunov-Schmidt reduction]
  \label{Lyapunov-Schmidt-reduction}
  Assume $\mathcal{J}_* \in C^2(E_2\setminus\{0\}; \R)$ and
  let $u_* \in E_2\setminus\{0\}$ be a non-de\-ge\-ne\-ra\-te critical point
  of $\mathcal{J}_*$.  There exists a neighborhood $U_*$ of $u_*$ in
  $H^1_0(\Omega)$ and a function
  $\gamma : \intervalco{2, 2+\ve} \to H^1_0(\Omega)$, $\ve > 0$,
  such that $\gamma(2) = u_*$ and
  \begin{equation*}
    \forall p \in \intervaloo{2, 2+\ve},\
    \forall u \in U_*,\qquad
    u \text{ solves } (\mathcal{P}_p)
    \iff
    u = \gamma(p).
  \end{equation*}
\end{thm}

\begin{prop}
  \label{symmetry-of-lens}
  Let $(u_p)_{p>2}$ be a family of l.e.n.s.\ converging to
  $u_* \in E_2\setminus\{0\}$.
  Assume $\mathcal{J}_* \in C^2(E_2\setminus\{0\}; \R)$ and
  $u_*$ is a non-degenerate critical point of $\mathcal{J}_*$.
  If $g u_* = g_*$
  (where $gu(x) := u(g^{-1}x)$ or $gu(x) := -u(g^{-1}x)$)
  for some $g \in O(N)$ such that
  $g\Omega = \Omega$ and $\mathcal{J}_p(gu) = \mathcal{J}_p(u)$ for
  all $u \in H^1_0(\Omega)$ and $p$, then, for $p$ close to $2$,
  $g u_p = u_p$.
\end{prop}

\begin{proof}
  If the conclusion does not hold, there exists a sequence $(u_{p_n})$
  of l.e.n.s.\ to~$(\mathcal{P}_{p_n})$ such that
  $\forall n,\ gu_{p_n} \ne u_{p_n}$.  Going if necessary to a
  subsequence, $u_{p_n} \wto u_* \in E_2$ and thus $u_{p_n} \to u_*$
  \cite[Theorem~4.5]{BBGS}.  In particular, for $n$ large enough,
  $(p_n, u_{p_n}) \in \intervalco{2, 2+\ve} \times U_*$ and so
  $u_{p_n} = \gamma(p_n)$.  But $gu_{p_n}$ is also a solution
  to~$(\mathcal{P}_{p_n})$ and $gu_{p_n} \to gu_* = u_*$.  Therefore,
  for $n$ large enough, $gu_{p_n} = \gamma(p_n) = u_{p_n}$ which is a
  contradiction.
\end{proof}

\begin{rem}
  The regularity assumption will be shown to hold for the square in
  section~\ref{sec:analysis-of-concavity}.
\end{rem}

\section{Reduction to a one-dimensional functional}
\label{sec.idea}

We sketch our strategy to prove Theorem \ref{main}.
Let $\Omega\subset\R^2$ be a square.  Without loss of generality we
can consider $\Omega = \intervaloo{0,1}^2$. In this case, a basis of
eigenfunctions of $E_2$ is given by
\begin{equation}
  \label{eq:basis:E2}
  \phi_1(x,y) = \sin(\pi x) \sin(2\pi y), \qquad
  \phi_2(x,y) = \sin(2\pi x) \sin(\pi y).
\end{equation}
Moreover, it easy to check that these functions are orthogonal in
$\Leb^2(\Omega)$, which allow us to parametrize $\tfrac{1}{2}\IS$ as
\begin{align} \label{param.s}
  \begin{split}
    \tfrac{1}{2} \IS
    &= \bigl\{u\in E_2: \|u\|_{\Leb^2(\Omega)} = \tfrac{1}{2} \bigr\}\\
    &= \bigl\{u_\theta(x,y):= \phi_1(x,y)\cos \theta - \phi_2(x,y) \sin \theta
    \ : \ \theta\in \intervalco{0,2\pi} \bigr\}.
  \end{split}
\end{align}
Hence, thanks to \eqref{eq:minimize:rS}, instead of \eqref{eq:minimize:S}
we can restrict ourselves to
study the minimization problem
\begin{equation} \label{eq.s.theta}
  \min_{\theta \in \intervalcc{0, 2\pi}} \mathcal{S}_*(u_\theta).
\end{equation}
Since $u_{\theta + \pi/2}(x,y) = u_\theta(y, -x)$,
$u_{\pi/4 - \theta}(x,y) = u_{\pi/4+\theta}(y, x)$
and $\mathcal{S}_*$ is invariant under the group of symmetries of the
square (the dihedral group of order~$8$)\label{S:symmetries},
we have that the
function $g(\theta) : \intervalcc{0,\tfrac{\pi}{4}} \to\R$ given by
\begin{equation} \label{entropy.f}
  g(\theta):=\mathcal{S}_*(u_\theta)
  =-\int_\Omega f\bigl(u_\theta(x,y)\bigr)\intd x \intd y,
  \hspace{1.1em} \text{with }
  f(t) =
  \begin{cases}
    t^2 \log|t|& \text{if } t \ne 0,\\
    0& \text{if } t = 0,
  \end{cases}
\end{equation}
is $\frac{\pi}{2}$-periodic and
$g(\tfrac{\pi}{4}-\theta)=g(\tfrac{\pi}{4}+\theta)$
(see Fig.~\ref{fig:graph:entropy}).

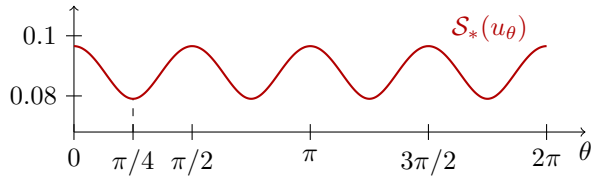
\begin{figure}[htb]
  \centering
  \newcommand{\ymin}{0.068}
  \begin{tikzpicture}[x=10mm, y=400mm]
    \draw[->] (0, \ymin) -- ++(6.8, 0) node[below]{$\theta$};
    \foreach \x/\l in {0/0, 0.785398/{\pi/4}, 1.570796/{\pi/2},
      3.14159/\pi, 4.712389/{3\pi/2}, 6.28318/{2\pi}}{
      \draw (\x, \ymin) ++(0, 3pt) -- ++(0, -6pt) node[below]{$\l$};
    }
    \draw[->] (0, \ymin) -- (0, 0.11);
    \foreach \y in {0.08, 0.1}{
      \draw (3pt, \y) -- (-3pt, \y) node[left]{$\y$};
    }
    \draw[dashed] (0.785398, \ymin) -- (0.785398, 0.0790503);
    \begin{scope}[color=red!70!black]
    \draw[thick] plot file{\graphpath{entropy.dat}};
    \node[above right] at (4.9, 0.095) {$\mathcal{S}_*(u_\theta)$};
    \end{scope}
  \end{tikzpicture}
  \caption{Graph of $\theta \mapsto \mathcal{S}_*(u_\theta)$.}
  \label{fig:graph:entropy}
\end{figure}

Observe that $u_{\pi/4}\in \frac{1}{2}\IS$ is even with respect to a
diagonal and odd with respect to the second one.
Because of the symmetries of $g$, it is clear that $\theta = \pi/4$ is
a critical point of $g$.
If we prove that the minimal energy in \eqref{eq.s.theta}
is achieved at $\theta = \pi/4$
and that the corresponding critical point of $\mathcal{J}_*$ is
non-degenerate, Theorem~\ref{symmetry-of-lens} will assert that
l.e.n.s.\ to \eqref{P} would
enjoy the symmetries of Conjecture~\ref{conj}.

\section{Ingredients of the computer-assisted proof}
\label{sec.proof}

In this section we prove that the function $\theta \mapsto g(\theta)$
reaches its
minimum at $\theta=\tfrac{\pi}{4}$.
Throughout this section $\Omega$ stands for the
square $\intervaloo{0,1}^2$.

The arguments in the proof can be sketched roughly as follows. We
design an adaptive integration method \code{integ} which allows to
integrate a function of two variables on $\Omega$.  By applying it to
$(x,y) \mapsto f(u_\theta(x,y))$, we obtain a function \code{entropy}
which
estimates $g$. With these tools, our task is reduced to
discard subintervals of $\intervalcc{0, \tfrac{\pi}{4}}$
which are guaranteed not to contain the minimum of~$g$.
This is performed by the function \code{exclude}.  The outcome is
a small interval around $\pi/4$ in which the minimum must lie.  Finally,
in that interval, we analyze the concavity of $g$ in order to guarantee
that it contains a unique critical point.  Consequently $\pi/4$ is the
only critical point in that interval and is therefore the minimum.

\medskip

When our scheme is implemented in a computer, one has to be
conscious of the existence of several sources of error involved in the
computations: numerical integration methods carry their corresponding
error, rounding errors come from computing with floating point
numbers, etc.  Hence, we have to track the precision of our
computations for our conclusions to be valid in spite of these various
errors.
Since in general real numbers
are not \emph{machine numbers}, to guarantee rigorous results when
performing computations, instead of considering \emph{values}, we will
consider \emph{intervals} containing the true values. In the next
subsection, for the reader's convenience,
we recall some basic facts about computing with floating
numbers and interval arithmetic.

\subsection{Interval arithmetic}

When implementations are made in digital electronic devices, due to
the limited space and speed, only a finite subset of real numbers is
available.  For the present considerations, we will use the set of
double precision floating point numbers that we will denote $\IF$.
Elementary arithmetic operations ($+$, $-$, $\cdot$, $/$) do not
necessarily return an element in $\IF$ even when their operands are in
$\IF$.  The standard IEEE-754 mandates that their implementation
returns a value in $\IF$ closest to the exact result.  This yields
small rounding errors which may propagate badly along the
computations.  A classical example~\cite{Rump-revisited,Rump88} is to evaluate
$f(x,y) = 333.75 y^6+x^2(11x^2 y^2-y^6 -121 y^4-2)+5.5 y^8$ at
$(x,y) = (77617,33096)$.
Despite the fact that all coefficients of $f$ and $x$, $y$
belong to $\IF$, when all computations are performed with
\emph{double precision}, the result is $-1.180...\cdot 10^{21}$ while
the correct value is $-2$.
Rounding errors may also affect the constants in the program: for
example $0.1 \notin \IF$ (because it has an infinite binary expansion)
and $\pi \notin \IF$ are rounded in computer memory.
Errors also come from the necessary approximation (think of the
truncation of an infinite sum) performed by the algorithm computing
functions such as $\sin$, $\cos$, $\log$,...

In order to account for all these errors along computations performed
on a computer, values $a \in \R$ will be represented by an interval
\begin{equation*}
  [a] = \intervalcc{\dl{a}, \ul a}
  = \{ x\in\R: \overline a \le x \le \underline a \}
\end{equation*}
(where $[a]$ must read as a single symbol) such that $a \in [a]$.
The set of intervals with endpoints in $\R$ will be denoted $\IIR$.
The \emph{width} of $[a]$ is $\ul{a} - \dl{a}$ and denoted $\width([a])$.
We also define $\low(\intervalcc{\dl{a}, \ul a}) := \dl{a}$
and $\high(\intervalcc{\dl{a}, \ul a}) := \ul{a}$.
Vectors $a \in \R^N$ will be represented as $[a] = \bigl(
[a_1], \dotsc, [a_N]) \in (\IIR)^N$ standing for the \emph{box}
$[a_1] \times \cdots \times [a_N]$.

Elementary arithmetic operations ($+$, $-$, $\cdot$, $/$) and in
general any function $f : \R^N \to \R^M$ must be extended to an interval
function $[f] : (\IIR)^N \to (\IIR)^M$ so that the following
\emph{containment property} holds:
\begin{equation*}
  \forall [a] \in \Dom [f],\qquad
  f([a]) = \bigl\{ f(x) : x \in [a] \cap \Dom f \bigr\}
  \subseteq [f]\bigl([a]\bigr).
\end{equation*}
This means that all exact values $f(x)$ for $x$ ranging in the
interval $[a]$ are contained in the interval returned by the function
$[f]$ on the operand $[a]$.
In this case, we will call $[f]$ an \emph{interval extension} of $f$.
Composition of interval extensions are again interval extensions.

When implemented on a computer, interval extensions of
$f : \R^N \to \R^M$ will naturally be functions $[f]$ defined on (a
subset of) $(\IIF)^N$ and returning values in $(\IIF)^M$ where $\IIF$
denotes the set of intervals with endpoints in~$\IF$.
Rounding errors of arithmetic operators will be handled by using
\emph{directed rounding modes} of the processor: for example, the
interval extension of the addition if $[a] \in \IIF$ and
$[b] \in \IIF$ will be given by
$\bigintervalcc{\rounddown{\dl{a} + \dl{b}}, \roundup{\ul{a} + \ul{b}}}$
where $\rounddown{\ldots}$ (resp.\ $\roundup{\ldots}$) means that the
outcome of operations within the braces if rounded towards
$-\infty$ (resp.\ $+\infty$).
For approximations made to evaluate functions, a theoretical error
bound must be derived in terms of computable quantities which will
then be used to provide an (over)estimation of the error on the
computer.

We refer the reader interested in more details about interval analysis
to \cite{book1,book2}.

\subsection{Interval extension of $\mathbf{u^2 \,\text{log}|u|}$}

One of the basic functions to implement for our task is
$u \mapsto u^2 \log\abs{u}$.  The naive extension is
$[u] \mapsto [u]^2 \cdot [\log]\bigabs{[u]}$ but it is unsuitable.
Indeed if $0 \in [u]$, the logarithm interval extension will return an
unbounded interval (of the form $\intervalcc{-\infty, a}$, which is
possible because $\pm\infty \in \IF$) and so will subsequent
computations.  Thus a more precise interval extension of this function is
necessary.  Because this function is even, one can suppose w.l.o.g.\
that $[u] \subseteq \intervalcc{0, +\infty}$.  In order to derive
interval bounds, one has to distinguish sub-intervals based on the
monotonicity of $u \mapsto u^2 \log u$ (which is decreasing on
$\intervalcc{0, \e^{-1/2}}$ and increasing after) and the sign of
$\log u$.  Thus one has first to compute an interval estimate
$\intervalcc{\dl{m}, \ul{m}} = [\exp]\bigl( \intervalcc{-0.5, -0.5}
\bigr)$ of the location of the minimum.  Then, for an interval of the
form $\intervalcc{0, \ul{u}}$ with $0 < \ul{u} < \dl{m}$, the
algorithm will return the interval
$\intervalcc{\rounddown{\roundup{\ul{u}^2} \cdot \log \ul{u}}, 0}$
(one has to round up the square because $\log \ul{u} < 0$).  A similar
analysis is done for other intervals $[u]$.
This more clever interval extension returns tight bounds for the
function and, in particular, returns bounded intervals even when $0
\in [u]$ (see Fig.~\ref{fig:t2logt}).

\begin{figure}[htb]
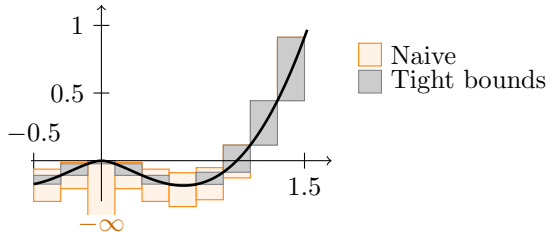

  \centering
  \begin{tikzpicture}[x=18mm, y=18mm]
    \begin{scope}[color=orange, fill=orange!10]
      \draw[fill] (1.9, 0.7) rectangle +(2ex, 2ex)
      node[right, yshift=-1ex, color=black]{Naive};
      \clip (-0.5, -0.4) rectangle (2, 1);
      \input{\graphpath{f_naive.tex}}
    \end{scope}
    \node[color=orange!80!black] at (0., -0.47) {$-\infty$};
    \begin{scope}[color=gray, fill=gray!40]
      \input{\graphpath{f.tex}}
      \draw[fill] (1.9, 0.5) rectangle +(2ex, 2ex)
      node[right, yshift=-1ex, color=black]{Tight bounds};
    \end{scope}
    \draw[->] (-0.52, 0) -- (1.7, 0);
    \draw[->] (0, -0.3) -- (0, 1.15);
    \draw (-0.5, 3pt) node[above]{$-0.5$} -- (-0.5, -3pt);
    \draw (1.5, 3pt) -- (1.5, -3pt) node[below]{$1.5$};
    \foreach \y in {0.5, 1}{
      \draw (3pt, \y) -- (-3pt, \y) node[left]{$\y$};
    }
    \draw[color=black, line width=1pt] plot file{\graphpath{f_float.dat}};
  \end{tikzpicture}
  \vspace{-1ex}
  \caption{Evaluation of $u \mapsto u^2 \log\abs{u}$.}
  \label{fig:t2logt}
\end{figure}

\subsection{Integration}

In order to compute $\mathcal{S}_*$, an integration on
the square $\Omega$ must
be performed.
There is a vast literature about integration methods, both for the
one-variable and multivariate cases, including for interval arithmetic
(see for example~\cite{Galdino12,Petras07,Petras07} and the
references therein).
For accurate evaluation of integrals in high dimensions, it is
recommended to avoid tensor products of one-dimensional rules (because
they require a number of function evaluations that is exponential in
the dimension) and instead turn to rules such as the one devised by
Smolyak~\cite{Smolyak63}.
In this paper however, the dimension is small and the fact that error
bounds of tensor product formula do not depend on mixed derivatives is
interesting.  Our integrand is also not smooth everywhere.

In this subsection, we briefly describe the adaptive verified
integration scheme that we are using.
Given an interval extension $[f]$ of a function $f : \Omega \to \R$,
we want to
return an interval $[\I_\Omega]$ such that\footnote{We choose to
  compute the integral mean because the weights in the sum
  approximating the integral are independent of the size of $\Omega$.}
\begin{equation*}
  [\I_\Omega] \ni
  \I_\Omega
  := \frac{1}{|\Omega|} \int_{\Omega}f(x,y)\intd x \intd y.
\end{equation*}
To compute $[\I_\Omega]$, a \emph{rule} that evaluates $[f]$ at
various points and estimates the error in terms of the variation (on
$\Omega$) of $f$ and its derivatives is used.  Such a rule returns
an interval $[\I_\Omega^0]$.  If we are happy with the width of
$[\I_\Omega^0]$, we take $[\I_\Omega] = [\I_\Omega^0]$ and stop.  If
not, we split $\Omega$ into four squares of equal sizes
$\Omega = \Omega_1 \cup \Omega_2 \cup \Omega_3 \cup \Omega_4$ and
recursively apply this procedure to have $[\I_{\Omega_i}]$,
$i = 1,\dotsc, 4$.  Then $[\I_\Omega]$ is defined~by
\begin{equation}
  \label{eq:split:integ}
  [\I_\Omega] := \frac{1}{4} \sum_{i=1}^4 [\I_{\Omega_i}]
  \ni \sum_{i=1}^4 \frac{\abs{\Omega_i}}{\abs{\Omega}} \,
  \I_{\Omega_i}
  = \I_\Omega .
\end{equation}
Since \eqref{eq:split:integ} implies
$\width([\I_\Omega]) \le \max_{i=1,\dotsc, 4} \width([\I_{\Omega_i}])$,
a natural stopping criteria for the recursion is that
$\width([\I_{\Omega_i}]) \le \ve$ where $\ve$ is a desired tolerance.

\subsubsection{Basic integration rule}

The simplest way to estimate $\I_\Omega$ is to use the following
simple fact (related to the mean value theorem for integrals): if
$f$ is integrable and
$\forall \xi \in \Omega,\linebreak[2]\ \dl{m} \le f(\xi) \le \ul{m}$,
then
\begin{equation*}
  \frac{1}{|\Omega|} \int_{\Omega}f(\xi)\intd \xi
  \in \intervalcc{\dl{m}, \ul{m}}.
\end{equation*}
If $\Omega \subseteq [x] \times [y]$, this is in particular true for
$\intervalcc{\dl{m}, \ul{m}} = [f]([x], [y])$.  Thus, one takes
\begin{equation*}
  [\I_\Omega^0] := [f]\bigl( [x], [y] \bigr).
\end{equation*}
The associated adaptive procedure does not perform very well however.
Indeed, we expect the width of $[f]\bigl( [x], [y] \bigr)$ to be of
size $\norm{\nabla f}_{\Leb^\infty(\Omega)} \diam\Omega$ where
$\diam\Omega$ is the diameter of $\Omega$.  Thus, if $h$ denotes the
diameter of the small squares obtained at depth $d$ of the recursion,
for $f \in C^{1}(\overline{\Omega})$, $\width( [\I_\Omega] )
= O(h)$.  In terms of the number $n$ of points at which the function
needs to be evaluated and in
terms of the depth $d$ of the recursion, this gives:
\begin{equation}
  \label{eq:order:basic}
  \width([\I_\Omega])
  = O(h)
  = O\Bigl( \frac{1}{\sqrt{n}} \Bigr)
  = O\Bigl( \frac{1}{2^d} \Bigr).
\end{equation}
This is quite slow.  For example, integrating the analytic function
$\phi_1$ (defined in~\eqref{eq:basis:E2}) using a recursion depth of~$14$
must perform at least 268\,468\,225 function evaluations for a final
precision of
about $10^{-3}$.  This is pictured in Fig.~\ref{fig:error-basic}.
Clearly, this is not practical as the procedure to determine
a small neighborhood of the minumum requires many evaluations of
$\mathcal{S}_*$ with a good precision.

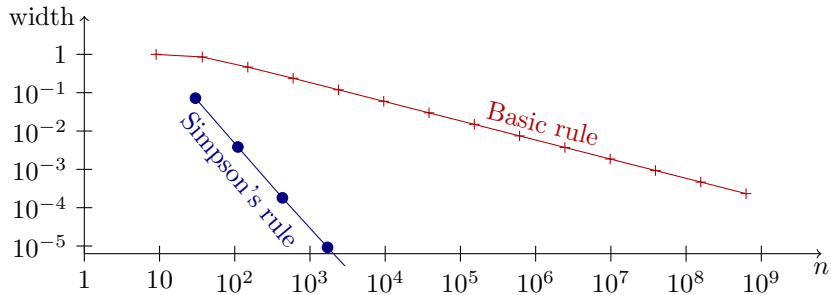
\begin{figure}[htb]
  \centering
  \newcommand{\ymin}{-5.2}
  \begin{tikzpicture}[x=10mm, y=5.1mm]
    \draw[->] (0, \ymin) -- (9.8, \ymin) node[below]{$n$};
    \draw[->] (0, \ymin) -- (0, 1) node[left]{width};
    \foreach \x/\l in {0/1, 1/10}{
      \draw (\x, \ymin) ++(0, 3pt) -- ++(0, -6pt) node[below]{$\l$};
    }
    \foreach \x in {2,...,9}{
      \draw (\x, \ymin) ++(0, 3pt) -- ++(0, -6pt) node[below]{$10^{\x}$};
    }
    \foreach \y/\l in {0/1, -1/{10^{-1}}, -2/{10^{-2}},
      -3/{10^{-3}}, -4/{10^{-4}}, -5/{10^{-5}}}{
      \draw (3pt, \y) -- (-3pt, \y) node[left]{$\l$};
    }
    \begin{scope}[color=red!70!black]
      \draw[mark=+] plot file{\graphpath{integ0_error.dat}};
      \node[rotate=-14] at (6.1, -1.8) {Basic rule};
    \end{scope}
    \begin{scope}[color=blue!50!black]
      \clip (0, -5.5) rectangle (10, 0);
      \draw[mark=*] plot file{\graphpath{integ4_error.dat}};
      \node[rotate=-50] at (2.05, -3.3) {Simpson's rule};
    \end{scope}
  \end{tikzpicture}
  \caption{Errors of adaptive integration rules.}
  \label{fig:error-basic}
\end{figure}

\subsubsection{Simpson's rule}
\label{sec:Simpson}

To locate the minimum of~\eqref{entropy.f} reasonably fast, we
evaluate $\mathcal{S}_*(u_\theta)$ at $\theta=\pi/4$
with a precision of $10^{-9}$.  A higher
order integration rule is therefore desirable.  The one we use is the
tensor product of two Simpson's rules.  Recall that the error of the
one-dimensional Simpson's rule applied to a function
$f \in C^4(\intervalcc{a, a+h}; \R)$ is given by
error~\cite[p.~288]{Davis-Rabinowitz84}:
\begin{equation*}
  \frac{1}{h} \int_a^{a+h} f(x) \intd x
  - \frac{1}{6} \bigl(f(a) + 4 f(a +\tfrac{1}{2}h) + f(a+h) \bigr)
  = - \frac{1}{2880} h^4 f^{(4)}(\xi)
\end{equation*}
for some $\xi \in \intervaloo{a, a+h}$.  This yields a
two-dimensional rule on a rectangle~$P$:
\begin{equation*}
  \text{Simpson}(f)
  = \frac{1}{36} \bigl(
  f_{00} + f_{20} + f_{02} + f_{22}
  + 4 (f_{10} + f_{01} + f_{21} + f_{12}) + 16 f_{11} \bigr)
\end{equation*}
where $f_{ij} := f(a_1 + i h_1/2, a_2 + j h_2/2)$,
and the following error bound:
\begin{equation}
  \label{eq:bound-simpson-2D}
  \biggl| \frac{1}{\abs{P}} \int_P f
  - \text{Simpson}(f) \biggr|
  \le \frac{1}{2880} \bigl( h_1^4 \norm{\partial_x^4 f}_\infty
  + h_2^4 \norm{\partial_y^4 f}_\infty \bigr).
\end{equation}
From this, an interval version is easily derived with convergence
speed
$\width [\I_P] = O(h^4) = O\bigl( \frac{1}{n^2} \bigr) = O\bigl(
\frac{1}{16^d} \bigr)$.

\subsubsection{Integrating a non-smooth function}

For the purpose of integrating $u_\theta^2 \log\abs{u_\theta}$, the
Simpson's rule cannot be used over the whole $\Omega$.
Indeed, whenever $u_\theta = 0$, the second and fourth derivatives
of the function are unbounded.
Thus, the following algorithm to estimate
$\frac{1}{|\Omega|} \int_{\Omega}f(x,y) \intd (x,y)$ is used.

\begin{algo}
  \code{mean\_integ($[f], [\partial^4f],
    [x_{\text{left}}], [x_{\text{right}}], [y_{\text{left}}], [y_{\text{right}}],
    d, \text{tol}$)}
  \normalfont
  \begin{itemize}
  \item Evaluate the derivatives $[\partial^4_{x}f]$,
    and $[\partial^4_{y}f]$ on the box
    $[x] \times [y] :=
    \bigintervalcc{\low([x_{\text{left}}]), \high([x_{\text{right}}])} \times
    \bigintervalcc{\low([y_{\text{left}}]), \high([y_{\text{right}}])}$
    and compute $\ul{e}$ using the right hand side
    of~\eqref{eq:bound-simpson-2D},
    rounding all operations towards~$+\infty$.
  \item let $[x_{\text{mid}}]$ (resp.\ $[y_{\text{mid}}]$) be the
    middle points of $[x_{\text{left}}]$ and $[x_{\text{right}}]$
    (resp.\ $[y_{\text{left}}]$ and $[y_{\text{right}}]$).
  \item If $\ul{e} \le \tfrac{1}{2} \code{tol}$, then let $[\I^0]$ be defined
    by~$\text{Simpson}(f) + [-\ul{e}, \ul{e}]$,
    otherwise set $[\I^0] := [f]\bigl([x], [y]\bigr)$.
  \item If $d \le 0 \lor \ul{e} \le \tfrac{1}{2}\code{tol}$, then
    return $[\I^0]$;\\
    else, call
    recursively \code{mean\_integ} on the four sub-squares
    (see Fig.~\ref{iteration.square}) with $d := d - 1$.  This will
    return $[\I^0_1], \dotsc,\linebreak[2] [\I^0_4]$.  Finally, in
    accordance with~\eqref{eq:split:integ}, return
    $\bigl([\I^0_1] + \cdots + [\I^0_4] \bigr) / 4$.
  \end{itemize}
\end{algo}

Note that, in the above algorithm, the parameter
\code{tol} is used as an indication of when
to stop the recursion,
the final accuracy being actually given by the width of the
returned interval.
Fig.~\ref{fig:integ-rules} depicts how this
algorithm subdivides the square $\intervalcc{0,1}^2$ to compute
$\int_{\intervalcc{0,1}^2} u_\theta^2 \log\abs{u_\theta}$ with
$\theta = \pi/4$ to achieve a tolerance of~$0.05$ with depth $d=5$.
A sub-square is
colored whenever the higher order formula was used on it.

The function \texttt{integ($[f], [\partial^4f], d, \text{tol}$)}
mentioned at the beginning of section~\ref{sec.proof} is simply defined
as
$\code{mean\_integ(}[f], [\partial^4f], \intervalcc{0,0},
\intervalcc{1,1}, \intervalcc{0,0}, \intervalcc{1,1}, d,
\text{\code{tol)}}$.

\begin{figure}[htb]
  \centering
  \begin{minipage}[b]{0.43\linewidth}
    \centering
    \newcommand{\xaxis}{-0.35}
    \begin{tikzpicture}[x=17mm, y=17mm]
      \draw (0,0) -- (1,0) -- (1,1) -- (0,1) -- cycle;
      \draw[red,dashed] (0,0.5) -- (1,0.5);
      \draw[red,dashed] (0.5,0) -- (0.5,1);
      \foreach \xy/\pos/\c/\l in {{0,0}/{below left}/blue/{00},
        {0.5,0}/{below}/red/{\text{m}0}, {1,0}/{below right}/blue/10,
        {1,0.5}/{right}/red/{1\text{m}}, {1,1}/{above right}/blue/{11},
        {0.5,1}/above/red/{\text{m}1},   {0,1}/{above left}/blue/01,
        {0,0.5}/left/red/{0\text{m}},
        {0.5,0.5}/{below right}/red/{\text{mm}}}{
        \fill[\c] (\xy) circle (1.7pt);
      }
      \draw[->] (\xaxis, -0.3) -- (1.35, -0.3);
      \draw[->] (\xaxis, -0.3) -- (\xaxis, 1.2);
      \foreach \x/\l in {0/left, 0.5/mid, 1/right}{
        \draw (\x, -0.3) ++(0, 3pt) -- ++(0, -6pt)
        node[below, inner sep=1pt]{\small $[x_{\text{\l}}]$};
        \draw (\xaxis, \x) ++(3pt, 0) -- ++(-6pt, 0)
        node[left, inner sep=1pt]{\small $[y_{\text{\l}}]$};
      }
    \end{tikzpicture}
    \hspace*{10pt}

    \caption{Square subdivisions.}
    \label{iteration.square}
  \end{minipage}
  \hfill
  \begin{minipage}[b]{0.56\linewidth}
    \centering
    \definecolor{light-green}{RGB}{0, 200, 0}
    \newcommand{\PrecisionAchieved}[4]{}
    \begin{tikzpicture}[x=25mm, y=25mm]
      \input{\graphpath{integration4.tex}}
    \end{tikzpicture}
    \par
    \caption{Using the appropriate integration rule
      with Simpson's rule.}
    \label{fig:integ-rules}
  \end{minipage}
\end{figure}

Now that verified integration is available, it is easy to define
$\code{entropy(} \theta, d, \code{tol)}$ that estimates $g(\theta)$.
Indeed it suffices to pass the function
$(x,y) \mapsto - (u_\theta(x,y))^2 \cdot
\log\abs{u_\theta(x,y)}$ and its fourth derivatives (which are
straightforward to compute as everything is explicit) to \code{integ}.

\subsection{Locating the minimum}

As we described before, our task is reduced to minimize the entropy
function $\mathcal{S}_*(u_\theta)$ on the sphere $\IS$.  Due to
symmetry considerations, it is enough to prove that the function
$g(\theta) = -\int_\Omega u_\theta^2 \log\abs{u_\theta}\intd x \intd y$
defined on $\intervalcc{0, \tfrac{\pi}{4}}$ attains its minimum at
$\theta=\tfrac{\pi}{4}$, where $u_\theta$ is defined in~\eqref{param.s}.
The symmetries also imply that $\frac{\pi}{4}$ is a critical point of $g$.

In this section, we determine a ``small'' interval around $\pi/4$ in
which the minimum is guaranteed to lie.  To that effect, we first
compute an interval
$\intervalcc{\dl{m}, \ul{m}} := [g]\bigl([\pi]/4\bigr) \in \IIF$
to have bounds on the value $g(\pi/4)$,
where $[\pi] \in \IIF$ denotes a small interval containing $\pi$.
Then, starting from the left of $\intervalcc{0, \pi/4}$, we try to
remove intervals in which we are sure the minimum does not lie.  More
precisely, an interval $[x]$ is discarded if the lower bound of
$[g]\bigl([x]\bigr)$ (which contains all $g(\xi)$ for $\xi \in [x]$)
is greater than $\ul{m}$.  This simple idea leads to the following
algorithm.

\begin{algo}[{exlude($[g], \dl{x}, \ul{x}, \text{step}, n$)}]
  if $n \le 0$ then return $\intervalcc{\dl{x}, \ul{x}}$\\
  else\\
  \begin{block}
    let $x_{\text{next}} = \min \{ \dl{x} + \text{step}, \ul{x} \}$\\[1pt]
    let $\intervalcc{\dl{y}, \ul{y}}
    = [g]\bigl(\intervalcc{\dl{x}, \ul{x}}\bigr)$\\
    if $\dl{y} > \ul{m}$ then return
    exclude($[g], x_{\text{next}}, \ul{x}, \text{step}, n - 1$)\\
    else return
    exclude($[g], \dl{x}, \ul{x}, \frac{1}{2}\text{step}, n - 1$)
  \end{block}
\end{algo}

\noindent
Remark that $n$ is the number of iterations to be performed.
In practice, the depth (and hence the precision) with which the
integral in $[g]$ is computed is gradually increased as $\dl{x}$ gets
closer to $\pi/4$.
This function is called with $\dl{x} = 0$ and $\ul{x} =
\high([\pi]/4)$ to obtain the location of the minimum.
This is illustrated on Fig.~\ref{fig:min-g} for $n = 80$ where the
gray rectangles are $[x] \times [g]\bigl([x]\bigr)$ with $[x]$
excluded.

\begin{figure}[htb]
  \centering
  \begin{tikzpicture}[x=50mm, y=700mm]
    \begin{scope}
      \definecolor{excluded}{gray}{0.7}
      \definecolor{taken}{RGB}{168, 0, 57}
      \input{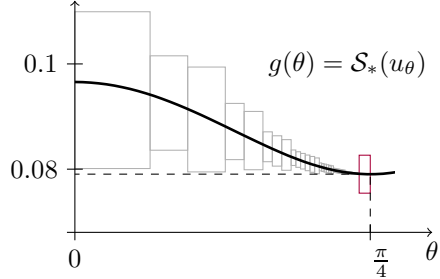}
    \end{scope}
    \begin{scope}
      \node[right] at (0.49, 0.1) {$g(\theta) = \mathcal{S}_*(u_\theta)$};
      \clip (0, 0.068) rectangle (0.85, 0.1);
      \draw[line width=1pt] plot file {\graphpath{entropy.dat}};
      \draw[dashed] (0.7854, 0.068) -- (0.7854, 0.07905016) -- (0, 0.07905016);
    \end{scope}
    \begin{scope}[shift={(0, 0.068)}]
      \draw[->] (-3pt, 0) -- (0.95, 0) node[below]{$\theta$};
      \draw[->] (0, -3pt) node[below]{$0$} -- (0, 0.044);
      \draw (0.7854, 3pt) -- +(0, -6pt) node[below right, xshift=-3pt]{%
        $\frac{\pi}{4}$};
    \end{scope}
    \foreach \y in {0.08, 0.1} {
      \draw (3pt, \y) -- (-3pt, \y) node[left]{$\y$};
    }
  \end{tikzpicture}
  \vspace{-2ex}
  \caption{The exclusion procedure for locating the minimum of $g(\theta)$.}
  \label{fig:min-g}
\end{figure}

\subsection{Analysis of concavity}
\label{sec:analysis-of-concavity}

The final step is to analyze the concavity of $g$ in the interval $I$
returned by \code{exlude} to guarantee that $\theta=\tfrac{\pi}{4}$ is
the unique critical point in~$I$, hence the only minimum.

\begin{prop}
  \label{prop:J''}
  The functional $\mathcal{J}_* : E_2 \to \R$ defined
  by~\eqref{j.star} is of class $C^2$ on $E_2\setminus\{0\}$
  and its second derivative
  $\partial_u^2 \mathcal{J}_*$ at $u \in E_2\setminus\{0\}$ is the
  bilinear map given by
  \begin{equation}
    \label{eq:J''}
    \partial_u^2 \mathcal{J}_*(u) [u_1, u_2]
    = -\lambda_2 \int_\Omega u_1 u_2 (1 + \log\abs{u}) \intd x \intd y.
  \end{equation}
  The function $g : \R \to \R$ defined by~\eqref{entropy.f} with
  $u_\theta$ being given in~\eqref{param.s} is $C^2$ and
  \begin{equation}
    \label{eq:g''}
    g''(\theta)
    = -2\left( 
    \frac{1}{4} + g(\theta) +  \int_\Omega |u_\theta'|^2 \log\abs{u_\theta}
      \intd x \intd y\right)
    \qquad
    \text{where } ' = \frac{\partial}{\partial \theta}.
  \end{equation}
\end{prop}

\begin{proof}
  \textit{1.}  The first derivative of $\partial_u \mathcal{J}_*$
  exists and is given by~\eqref{eq:J.star'}.  It is sufficient to show
  that the map
  $(t, \theta) \mapsto \partial_u \mathcal{J}_*(tu_\theta)$ is $C^1$.
  Indeed, first notice that
  $\mathcal{S} : \intervaloo{0, +\infty} \times \R \to
  E_2\setminus\{0\} : (t, \theta) \mapsto t u_\theta$ is surjective.
  Second, one readily shows that, for all $(t, \theta)$, the
  differential $\partial_{(t,\theta)} \mathcal{S} (t,\theta)$ is an
  one-to-one linear map so the inverse function theorem says that
  $\mathcal{S}$ is locally invertible as a $C^1$-map which proves the
  differentiability of $\partial_u \mathcal{J}_*$ on
  $E_2\setminus\{0\}$.  Since
  $\partial_u \mathcal{J}_*(tu)[v] = t \, \partial_u
  \mathcal{J}_*(u)[v] -\lambda_2 t \log t \int uv$, the derivative
  w.r.t.\ $t$ clearly exists and is continuous.  It remains to show
  that the derivative w.r.t.\ $\theta$ exists, is given by
  \begin{equation}
    \label{eq:J''theta:t}
    \partial_\theta\bigl( \partial_u \mathcal{J}_*(tu_\theta) \bigr) [v]
    = -\lambda_2 \int_\Omega v \, tu'_\theta (1 + \log\abs{tu_\theta})
    \intd x \intd y
  \end{equation}
  and is continuous.  Given the above formula for
  $\partial_u \mathcal{J}_*(tu)$, it is enough to
  prove~\eqref{eq:J''theta:t} for $t = 1$.  From the form of both
  partial derivatives as well as $\partial_{(t,\theta)} \mathcal{S}$,
  one readily deduces~\eqref{eq:J''}.

  \smallskip
  \noindent
  \textit{2.}~Remark that the function $f : \R \to \R$ defined
  in~\eqref{entropy.f} is $C^1$ and $u_\theta$ and $u_\theta'$ are
  bounded functions of $(x,y)$.  Thus Lebesgue dominated convergence
  Theorem implies that $g$ is continuously differentiable and
  \begin{equation*}
    g'(\theta)
    = - \int_\Omega \partial_tf(u_\theta) u_\theta' \intd x \intd y
    = - \int_\Omega u_\theta' u_\theta (1 + 2 \log\abs{u_\theta})
    \intd x \intd y .
  \end{equation*}
  A direct computation shows that
  $\int_\Omega u'_\theta u_\theta = 0$, so $g'$ may be rewritten as
  \begin{equation}
    \label{eq:g'}
    g'(\theta)
    = -2 \int_\Omega u_\theta' u_\theta \log\abs{u_\theta} \intd x \intd y.
  \end{equation}
  Notice this is essentially
  $\partial_u \mathcal{J}_*(u_\theta)[u'_\theta]$ so, if we
  establish~\eqref{eq:J''theta:t}, \eqref{eq:g''} will easily follow by
  the chain rule and the facts: $u_\theta'' = -u_\theta$ and
  $\norm{u_\theta'}_{\Leb^2(\Omega)} = 1/2$.

  \smallskip
  \noindent
  \textit{3.}~We now establish~\eqref{eq:J''theta:t} for $t = 1$.  Using
  the mean value Theorem, the differential quotient may be written as
  \begin{equation*}
    \frac{\partial_u \mathcal{J}_*(u_{\theta+\ve})[v]
      - \partial_u \mathcal{J}_*(u_\theta)[v]}{\ve}
    = -\lambda_2 \int_\Omega v u'_{\theta_\ve} (1 +
    \log\abs{u_{\theta_\ve}}) \intd x \intd y
  \end{equation*}
  for some $\theta_\ve \in \intervaloo{\theta, \theta + \ve}$, so the
  theorem will be proved if we show that the map
  \begin{equation}
    \label{eq:J''theta}
    \theta \mapsto
    \int_\Omega \varphi(\theta, x, y) \intd x\intd y,
    \quad\text{where }
    \varphi(\theta, x, y) :=
    v \, u'_\theta \, (1 + \log\abs{u_\theta}),
  \end{equation}
  is well defined and continuous for all $v \in E_2$.  To prove this,
  the Lebesgue dominated convergence Theorem cannot be applied
  directly because, as $\theta$ varies, the nodal lines $\{ (x,y) :
  u_\theta(x,y) = 0 \}$ cover a non-zero measure set making any
  dominating function non-integrable.

  We will show the continuity of~\eqref{eq:J''theta} for
  $\theta \in \intervalcc{0, \pi/4}$ as, thanks to symmetries (see
  p.~\pageref{S:symmetries}), the remaining range of $\theta$ is very
  similar.  First note that $u_\theta$ may be written as
  $u_\theta(x,y) = z(x,y) \,  w_\theta(x,y)$ with
  \begin{equation*}
    z(x,y) := 2 \sin(\pi x) \sin(\pi y)
    \qquad\text{and}\qquad
    w_\theta(x,y) :=
    \cos \theta \cos(\pi y) - \sin \theta \cos(\pi x).
  \end{equation*}
  The nodal line of $u_\theta$ is the subset of $\Omega$ where
  $w_\theta$ vanishes.  For $\theta \in \intervalcc{0, \pi/4}$, it is
  the graph of a function
  $r_\theta : \intervalcc{0,1} \to \intervalcc{0,1} : x \mapsto
  r_\theta(x)$ defined implicitly by
  \begin{equation}
    \label{eq:r:theta}
    \forall x \in \intervalcc{0,1},\qquad
    \cos \theta \cos(\pi r_\theta(x)) - \sin \theta \cos(\pi x) = 0.
  \end{equation}
  Let us split $\Omega$ as the union of $\Omega^+$ and $\Omega^-$
  where
  \begin{align*}
    \Omega^+ &:= \{ (x,y) : x \in \intervaloo{0,1} \text{ and }
               0 < y \le r_\theta(x) \},
    \\
    \Omega^- &:= \{ (x,y) : x \in \intervaloo{0,1} \text{ and }
               r_\theta(x) < y < 1 \}.
  \end{align*}
  Thanks to the oddness of $u_\theta$ and $u_\theta'$ w.r.t.\ to the
  center of $\Omega$, we only have to consider $\Omega^+$, the case of
  $\Omega^-$ being similar.  One has
  \begin{equation*}
    \begin{aligned}
      \int_{\Omega^+} \varphi(\theta, x, y) \intd y\intd x
      &= \int_0^1 \int_0^{r_\theta(x)} \varphi(\theta, x, y) \intd y\intd x\\
      &= \int_0^1 \int_0^1 \varphi(\theta, x, s r_\theta(x)) r_\theta(x)
      \intd s \intd x.
    \end{aligned}
  \end{equation*}
  If we find an integrable function (independent of $\theta$)
  dominating $\varphi(\theta, x, s r_\theta(x)) r_\theta(x)$, the
  proof will be done.  Note that, because $\norm{u_\theta'}_\infty \le
  2$,
  \begin{equation*}
    \bigabs{\varphi(\theta, x, s r_\theta(x)) r_\theta(x)}
    \le 2 \norm{v}_\infty
    \bigl(1 + \bigabs{\log \abs{u_\theta(x, s r_\theta(x))}}\bigr)
  \end{equation*}
  so it suffices to find an integrable function dominating
  $\bigabs{\log \abs{u_\theta(x, s r_\theta(x))}}$.  Now, since
  $u_\theta \ge 0$ on $\Omega^+$ and $u_\theta \le 2$, it remains to
  find a lower bound of $u_\theta$ whose logarithm is integrable.
  Observe that $\log u_\theta = \log z + \log w_\theta$.  Since
  $z(x,y) \ge 2x(1-x) y(1-y)$, $z$ in integrable on $\Omega$, whence
  on~$\Omega^+$.  For $w_\theta$, using~\eqref{eq:r:theta}, one has
  \begin{align*}
    w_\theta(x, s r_\theta(x))
    &= \cos(\theta)
      \bigl(\cos(\pi s r_\theta(x)) - \cos(\pi r_\theta(x)) \bigr) \\
    &= \cos(\theta)  \pi r_\theta(x)
      \int_s^1 \sin(\pi \sigma r_\theta(x)) \intd \sigma\\
    &\ge \cos(\theta) \pi r_\theta(x)
      \int_s^1 \pi\, \sigma r_\theta(x) (1 - \sigma r_\theta(x)) \intd\sigma
    && \text{(as } \sigma r_\theta \in \intervalcc{0,1} \text{)}\\
    &= \cos(\theta) \pi^2 r^2_\theta(x) (1 - s)
      \Bigl( \frac{1 + s}{2} - \frac{1 + s + s^2}{3} r_\theta(x)
      \Bigr)  \\
    &\ge \cos(\theta) \pi^2 r^2_\theta(x) (1 - s)
      \Bigl( \frac{1 + s}{2} - \frac{1 + s + s^2}{3} \Bigr)
    && \text{(as } r_\theta \le 1 \text{)}\\
    &\ge \tfrac{1}{6} \cos(\theta) \pi^2 r^2_\theta(x) (1 - s)^2
    && \text{(for } s \in \intervalcc{0,1} \text{)} \\
    &\ge \tfrac{\sqrt{2}}{12} \pi^2 \,
      \min\{x^2, \tfrac{1}{4}\} \,(1 - s)^2
    &&  \text{(as } \theta \in \intervalcc{0, \tfrac{\pi}{4}} \text{).}
  \end{align*}
  The logarithm of this last bound is easily seen to be an integrable
  function of $(x, s) \in \intervaloo{0,1}^2$, thereby concluding the
  proof.
\end{proof}

In view of~\eqref{eq:g''}, to check the concavity of $g$, the
following condition must be verified:
\begin{equation}
  \label{cond.2d}
  g''(\theta)>0 \liff h(\theta) > \tfrac{1}{4} + g(\theta),
  \qquad \text{where }
  h(\theta)
  := -\int_\Omega |u_\theta'|^2 \log\abs{u_\theta} \intd x \intd y.
\end{equation}
Because the integrand of $h$ is singular whenever $u_\theta = 0$,
when evaluated with interval arithmetic, $[\log]$ will return
intervals of the form $\intervalcc{-\infty, \ul{y}} \in \IIF$ which
will contaminate the sum computing the integral.  The important remark
however is that the singularity helps $h$ to be large and so does not
need to be controlled.  Interval extensions $[h]$ of $h$ will return
intervals of the form $\intervalcc{\dl{y}, +\infty}$ and, to check
condition~\eqref{cond.2d}, one will verify that
$\dl{y} > \high\bigl(1 + [g]([\theta])\bigr)$.

\subsection{Non-degeneracy}

\begin{prop}
  \label{transfer-non-degeneracy}
  If $\theta_*$ is a non-degenerate critical point of $g$, then
  $u_* = t^*_{u_{\theta_*}} u_{\theta_*}$, where $t^*_u > 0$ is
  defined by~\eqref{eq:proj:N*}, is a non-degenerate critical
  point of $\mathcal{J}_*$.
\end{prop}

\begin{proof}
  Let $\theta_*$ be a critical point of~$g$.  On one hand, by
  definition~\eqref{eq:proj:N*} of $t^*_{u_{\theta_*}}$, one has
  $\partial_u \mathcal{J}_*(u_*)[u_{\theta_*}] = 0$.  On the other
  hand, \eqref{eq:g'} implies that
  $\partial_u \mathcal{J}_*(u_{\theta_*})[u_{\theta_*}'] = 0$.  Since
  $\partial_u \mathcal{J}_*(u_{\theta_*})$ vanishes in two orthogonal
  directions, $u_*$ is a critical point of $\mathcal{J}_*$.

  In view of proposition~\ref{prop:J''},
  \begin{equation*}
    \partial_u^2 \mathcal{J}_*(u_*)[u_*, v]
    = \partial_u^2 \mathcal{J}_*(u_*)[u_*, v]
    - \partial_u \mathcal{J}_*(u_*)[v]
    = -\lambda_2 \int_\Omega u_* v .
  \end{equation*}
  In particular,
  $\partial_u^2 \mathcal{J}_*(u_*)[u_*, u_*] = -\lambda_2 \int_\Omega
  u_*^2 < 0$ and
  $\partial_u^2 \mathcal{J}_*(u_*)[u_*, u'_{\theta_*}] = 0$ (recalling
  that $\int u_{\theta_*} u'_{\theta_*} = 0$).  Using again the
  definition~\eqref{eq:proj:N*} of $t^*_{u_{\theta_*}}$, one obtains
  \begin{equation*}
    \partial_u^2 \mathcal{J}_*(u_*)[u'_{\theta_*},u'_{\theta_*}]
    = \frac{\lambda_2}{2} g''(\theta_*) \ne 0.
  \end{equation*}
  As a consequence, the bilinear form $\partial_u^2
  \mathcal{J}_*(u_*)$ is non-degenerate.
\end{proof}

Consequently, our main result is a consequence of the following
proposition.

\begin{prop}
  \label{prop.main}
  Let $\Omega= \intervaloo{0,1}^2$.
  Then the minimization problem \eqref{eq.s.theta} is solved by
  $\theta = \pi/4$ and the minimum is non-degenerate.  Therefore
  the minimum of $\mathcal{J}_*$ on $\mathcal{N}_*$ is achieved by a
  multiple of $u_{\pi/4}$ and it is a non-degenerate critical point.
\end{prop}

\section{Proof: putting it all together}
\label{sec.main}

In this section, we finally run the machinery described above
on a computer and give the numerical results.  We
will sometimes write intervals as, for example,
$\compactinterval{1.74}{67}{81}$ instead of $[1.7467, 1.7481]$
to make easier to spot the common digits.

First, to locate an interval
$[\theta_{\text{min}}] \ni \tfrac{\pi}{4}$ where the entropy function
\eqref{entropy.f} attains its minimum, an small interval $[\pi]/4$
containing $\pi/4$ is determined and $[g]\bigl([\pi]/4\bigr)$ is
estimated with good precision by calling
\code{entropy($[\pi]/4$, depth, tol)} with $\code{depth} = 13$ and
$\code{tol} = 10^{-9}$.  The result is
\begin{equation*}
  [g]\bigl([\pi]/4\bigr) = \compactinterval{0.07905032}{01}{53}.
\end{equation*}
Then \code{exclude} with $n = 100$ is run on $[g]$ and outputs the
desired interval $[\theta_{\text{min}}]$:
\begin{equation*}
  [\theta_{\text{min}}] = \compactinterval{0.7}{51367}{85398}.
\end{equation*}
Finally, in order to verify that $\tfrac{\pi}{4}$ is the unique
critical point in $[\theta_{\text{min}}]$, we analyze the second
derivative of \eqref{entropy.f} in $[\theta_{\text{min}}]$ by checking
condition \eqref{cond.2d}.  Note that, if this condition is fulfilled,
our task is done.  We evaluate $[h]$ by
using \code{integ} with a depth $d = 8$ and tolerance
$\code{tol} = 10^{-3}$ and get:
\begin{equation*}
  [h]\bigl([\theta_{\text{min}}]\bigr)
  = \intervalcc{0.35179120, +\infty} .
\end{equation*}
Computing the right hand side of~\eqref{cond.2d} yields:
\begin{equation*}
  \tfrac{1}{4} + [g]\bigl([\theta_{\text{min}}]\bigr)
  = \intervalcc{0.324879, 0.333196}.
\end{equation*}
Clearly, these results show that condition~\eqref{cond.2d} is
satisfied, thereby proving Theorem~\ref{main}.

\bibliographystyle{amsplain}
\bibliography{biblio}

\end{document}